\pdfoutput=1
\documentclass{article}
\usepackage[utf8]{inputenc}
\usepackage{amsmath}
\usepackage{amssymb}
\usepackage{amsthm}
\usepackage{tikz-cd}
\usepackage{mathtools}
\usepackage{hyperref}
\usepackage[totalwidth=480pt, totalheight=680pt]{geometry}

\title{Orlov's Theorem in the Smooth Proper Case}
\author{Noah Olander}
\date{}

\newtheorem*{thm}{Theorem}
\newtheorem*{lemmaa}{Lemma A}
\newtheorem*{lemmab}{Lemma B}
\newtheorem*{lemmac}{Lemma C}
\newtheorem{prop}{Proposition}
\newtheorem{lemma}{Lemma}

\begin{document}

\maketitle

\begin{abstract}
    We extend Orlov's result that certain functors between derived categories of smooth projective varieties are Fourier--Mukai transforms to the case when those varieties are smooth and proper.
\end{abstract}

\section{Introduction}
In the landmark paper \cite{orlov1997equivalences}, Orlov proved that for smooth projective varieties $X, Y$ over a field $k$, a fully faithful, $k$-linear, exact functor $D^b_{coh}(X) \to D^b_{coh}(Y)$ having a right adjoint is isomorphic to a Fourier--Mukai transform. In this paper, we will show how to use modifications of Orlov's argument introduced by Ballard, de Jong, and this author to prove Orlov's Theorem in the smooth proper case:

\begin{thm}
Let $X, Y$ be smooth proper varieties over a field $k$. Let
$F : D^b_{coh} (X) \to D^b_{coh} (Y)$ be a fully faithful functor which is exact and $k$-linear. Then there is an object $E$  of $D^b_{coh} (X \times Y)$ such that $F$ is isomorphic to the functor
\begin{equation*}
    \Phi_E (K) = \mathbf{R} pr_{2 *} (\mathbf{L}pr_1^*K \otimes ^{\mathbf{L}}_{\mathcal{O}_{X \times Y}} E).
\end{equation*}
\end{thm}

We find the modifications to be pleasantly natural, so that what shines through is the power and elegance of Orlov's original proof. We also have access to the results of Bondal--Van den Bergh on strong generators and existence of adjoints (\cite{bondal2003}), which help simplify our statements and our proofs. The paper is organized as follows:

In Section 2, we show how to produce an object $E \in D^b_{coh} (X \times Y)$. Orlov's construction proceeds by embedding $X$ in projective space and then integrating a complex obtained from the Beilinson resolution of the diagonal. In \cite{2009arXiv0905.3148B}, Ballard shows how to replace the Beilinson resolution of the diagonal of $\mathbf{P}^N \times \mathbf{P}^N$ with any resolution of the diagonal of $X \times X$ by box products of vector bundles, thus immediately producing an object of $D^b_{coh} (X \times Y).$ There is no difficulty extending Ballard's ideas from the quasi-projective case to the smooth proper case since a smooth proper variety has enough vector bundles.

Having produced a kernel $E$, we conclude Section 2 by asking for which coherent $\mathcal{O}_X$-modules $\mathcal{F}$ is $F(\mathcal{F}) \cong \Phi_E(\mathcal{F})$. Orlov uses vanishing of cohomology to prove this for all powers of a very ample invertible sheaf $\mathcal{O}_X(1),$ and actually produces natural isomorphisms in this case. Not having ample line bundles at our disposal, we instead follow an idea of de Jong and show that $F$ and $\Phi_E$ agree on the spanning class of skyscraper sheaves of points.

In sections 3 and 4, we show that in fact $F$ and $\Phi_E$ differ only by tensor product with a line bundle on $X$. Since tensoring with a line bundle is a Fourier--Mukai functor and compositions of Fourier--Mukai functors are Fourier--Mukai, this completes the proof. The task is split into two separate parts: First we show in Section 3 that this is true when we restrict to the category of coherent sheaves. Then in Section 4 we extend this to the entire derived category. Section 4 should be compared with the proof of Proposition 2.16 in \cite{orlov1997equivalences} from which our section is taken almost verbatim. The key difference is that while Orlov uses an ample sequence $\mathcal{O}_X(-n)$, we use the the \emph{sequences} $\mathcal{O}_X (-nD_i)$ where $D_i$ are the complements of finitely many affine opens $U_i$ which cover $X$. We are able to make this work by exploiting the vanishing of
\begin{equation*}
    \mathrm{colim}_{n \geq 0}H^j (X, \mathcal{F}(nD_i))
\end{equation*}
for $j>0$, $\mathcal{F}$ coherent as a replacement for actual Serre vanishing.

The author would like to warmly thank Johan de Jong for asking him this question and for being so enthusiastic and available for mathematical discussions, especially during these strange times. This work would not have been possible without his idea to use points instead of negative line bundles.

\section{Producing a Kernel}

We produce here a candidate for a kernel of our functor. This section is basically an exposition of the construction in \cite[\href{https://stacks.math.columbia.edu/tag/0G07}{Tag 0G07}]{stacks-project}, so some parts are sketched. The outcome is summarized in Proposition 1, which is all that is needed for the sequel.

Let $X, Y$ be smooth proper varieties over a field $k$, and let $F : D^b_{coh} (X) \to D^b_{coh} (Y)$ be a fully faithful, exact, $k$-linear functor. The following Lemma will play a similar role in our construction to the role Lemma 2.4 of \cite{orlov1997equivalences} plays in Orlov's.
\begin{lemma}
$F$ is bounded: There is $m>0$ such that for every coherent sheaf $\mathcal{F}$ on $X$, $F(\mathcal{F})$ has nonzero cohomology sheaves only in degrees $[-m, m].$
\end{lemma}
\begin{proof}
This is proved in \cite[\href{https://stacks.math.columbia.edu/tag/0FZ8}{Tag 0FZ8}]{stacks-project}. It is a consequence of Bondal--Van den Bergh's theorem that $D^b_{coh} (X)$ has a strong generator. See \cite{bondal2003}.
\end{proof}
Since $X$ has enough vector bundles (see Lemma A in Section 4), a very general argument shows that every coherent sheaf on $X \times X$ is the quotient of a vector bundle $\mathcal{E} \boxtimes \mathcal{F}$ where $\mathcal{E}$, $\mathcal{F}$ are vector bundles on $X$ (\cite[\href{https://stacks.math.columbia.edu/tag/0FZ0}{Tag 0FZ0}]{stacks-project}). Therefore, we may construct a resolution of the structure sheaf of the diagonal $\Delta \subset X \times X$ by such vector bundles:
\begin{equation}
    \cdots \to \mathcal{E}_n \boxtimes \mathcal{F}_n \to \cdots \to \mathcal{E}_0 \boxtimes \mathcal{F}_0 \to \mathcal{O}_\Delta \to 0
\end{equation}
Then by applying $F$ to the second box tensor factors, we obtain a complex
\begin{equation}
    \cdots \to \mathcal{E}_n \boxtimes F(\mathcal{F}_n) \to \cdots \to \mathcal{E}_0 \boxtimes F(\mathcal{F}_0)
\end{equation}
in $D^b_{coh} (X \times Y)$ (the pullbacks and tensor products defining the box products in (2) are derived). By the K\"unneth formula, 
\begin{align*}
    \mathrm{Ext}^k _{\mathcal{O}_{X\times Y}}(\mathcal{E}_i \boxtimes F(\mathcal{F}_i), \mathcal{E}_j \boxtimes F(\mathcal{F}_j))
    &= \bigoplus _{m + n = k} \mathrm{Ext}^m_{\mathcal{O}_X} (\mathcal{E}_i , \mathcal{E}_j) \otimes _k \mathrm{Ext}^n_{\mathcal{O}_Y} (F(\mathcal{F}_i), F(\mathcal{F}_j)) \\
    &= \bigoplus _{m + n = k} \mathrm{Ext}^m_{\mathcal{O}_X} (\mathcal{E}_i , \mathcal{E}_j) \otimes _k \mathrm{Ext}^n_{\mathcal{O}_X} (\mathcal{F}_i, \mathcal{F}_j),
\end{align*}
where we have used fully faithfulness for the second equality. This vanishes if $k<0,$ so there is no obstruction to integrating the complex (2) to a right Postnikov system:

\begin{equation}
    \begin{tikzcd}[column sep=tiny]
         \cdots & &E_2 \ar[ll, "+1"'] \ar[dr] & &E_1 \ar[ll, "+1"'] \ar[dr] \ar[dr] & &E_0 \ar[ll, "+1"'] \ar[dr, "\cong"] &\\
         & \cdots \ar[rr] & & \mathcal{E}_2 \boxtimes F(\mathcal{F}_2) \ar[rr] \ar[ur] & & \mathcal{E}_1 \boxtimes F(\mathcal{F}_1) \ar[rr]\ar[ur] & & \mathcal{E}_0 \boxtimes F(\mathcal{F}_0)
    \end{tikzcd}
\end{equation}
The triangles with an arrow of degree $+1$ are distinguished and the other triangles commute. The objects $E_n$ are unique up to non-unique isomorphism. See \cite[\href{https://stacks.math.columbia.edu/tag/0D7Y}{Tag 0D7Y}]{stacks-project}. Note that we use right Postnikov systems instead of left systems as in \cite{orlov1997equivalences}. This is because our complexes are unbounded to the left. 

We also have a trivial Postnikov system
\begin{equation}
    \begin{tikzcd}[column sep=tiny]
         \cdots & &K_2 \ar[ll, "+1"'] \ar[dr] & &K_1 \ar[ll, "+1"'] \ar[dr] \ar[dr] & &K_0 \ar[ll, "+1"'] \ar[dr, "\cong"] &\\
         & \cdots \ar[rr] & & \mathcal{E}_2 \boxtimes \mathcal{F}_2 \ar[rr] \ar[ur] & & \mathcal{E}_1 \boxtimes \mathcal{F}_1 \ar[rr]\ar[ur] & & \mathcal{E}_0 \boxtimes \mathcal{F}_0,
    \end{tikzcd}
\end{equation}
with $K_n$ the $n^{th}$ naive truncation of the complex (1). Thus $K_n$ has cohomology sheaves only in degrees 0 and $-n$. Since $X\times Y$ has finite homological dimension, for $n\gg 0$ we have $K_n \cong \mathcal{O}_\Delta \oplus \mathcal{H}^{-n} (K_n)[n],$ and $\mathcal{H}^{-n}(K_n)$ is a locally free $\mathcal{O}_{X \times Y}$-module. Our task in the remainder of the section is to show that $E_n$ has a similar decomposition $E \oplus C_n$ where $E$ is the desired kernel and $C_n$ slides off to $- \infty$.

Let us compute $\Phi_{E_n} (\mathcal{F})$ for $\mathcal{F}$ a coherent $\mathcal{O}_X$-module whose support has dimension 0: Apply the functor
\begin{equation*}
    K \mapsto \mathbf{R}pr_{2 *} (\mathbf{L}pr_1^*(\mathcal{F}) \otimes _{\mathcal{O}_{X \times Y}}^{\mathbf{L}} K)
\end{equation*}
to the entire diagram (3) to see that $\Phi_{E_n} (\mathcal{F})$ is an $n^{th}$ convolution of the complex
\begin{equation}
\Gamma (X, \mathcal{F} \otimes \mathcal{E}_{\bullet}) \otimes _k F(\mathcal{F}_\bullet)
\end{equation}
(the functor $\Gamma$ is underived since $\mathcal{F}$ has $0$-dimensional support).

On the other hand, we can apply the functor
\begin{equation*}
    K \mapsto F(\mathbf{R} pr_{2 *} (\mathbf{L}pr_1^*\mathcal{F} \otimes _{\mathcal{O}_{X \times X}}^{\mathbf{L}} K))
\end{equation*}
to the entire diagram (4). This transforms the second row of (4) into the complex (5). Thus, for $n \gg 0$, another $n^{th}$ convolution of the complex (5) is
\begin{align*}
    F(\mathbf{R} pr_{2 *} (\mathbf{L}pr_1^*\mathcal{F} &\otimes _{\mathcal{O}_{X \times X}}^{\mathbf{L}} K_n)) \\ & = F(\mathbf{R} pr_{2 *} (\mathbf{L}pr_1^*\mathcal{F} \otimes _{\mathcal{O}_{X \times X}}^{\mathbf{L}} \mathcal{O}_{\Delta})) \oplus F(\mathbf{R} pr_{2 *} (\mathbf{L}pr_1^*\mathcal{F} \otimes _{\mathcal{O}_{X \times X}}^{\mathbf{L}} \mathcal{H}^{-n}(K_n)))[n] \\ & = F(\mathcal{F}) \oplus F(\mathcal{K}_n)[n],
\end{align*}
where $\mathcal{K}_n$ is a coherent sheaf on $X$.

Therefore, by uniqueness of convolutions (note that negative Ext's vanish between the terms of the complex (5) since $F$ is fully faithful), we obtain isomorphisms
\begin{equation*}
    \Phi_{E_n} (\mathcal{F}) \cong F(\mathcal{F}) \oplus F(\mathcal{K}_n)[n].
\end{equation*}
By Lemma 1, we deduce that there is $m>0$ so that for every $\mathcal{F}$ coherent with 0-dimensional support and $n \gg 0$, $\Phi_{E_n} (\mathcal{F})$ has cohomology sheaves only in degrees $[-n-m, -n+m] \cup [-m, m].$ But knowing this for every such $\mathcal{F}$ implies that $E_n$ has cohomology sheaves only in degrees $[-n-m, -n+m] \cup [-m, m]$ (see \cite[\href{https://stacks.math.columbia.edu/tag/0FZ9}{Tag 0FZ9}]{stacks-project} for details). Thus if we take $n$ to be very large, since $X \times Y$ has finite homological dimension, we will have a decomposition $E_n = E \oplus C_n,$ where $E$ has cohomology sheaves in degrees $[-m,m]$ and $C_n$ has cohomology sheaves in degrees $[-n-m, -n+m].$ Then it follows that $\Phi_E(\mathcal{F}) \cong F(\mathcal{F})$ and $\Phi_{C_n} (\mathcal{F}) \cong F(\mathcal{K}_n)[n].$ Thus we have proved:

\begin{prop}
Let $X, Y$ be smooth proper varieties over a field $k$. Let $F : D^b_{coh} (X) \to D^b_{coh} (Y)$ be a fully faithful functor which is exact and $k$-linear. Then there is an object $E$  of $D^b_{coh} (X \times Y)$ such that $F(\mathcal{F}) \cong \Phi_E(\mathcal{F})$ for every coherent sheaf $\mathcal{F}$ on $X$ whose support has dimension 0.
\end{prop}

\section{Fully Faithful Functors Agreeing on Points}

The following lemma should be compared with Lemma 2.15 in \cite{orlov1997equivalences}.

\begin{lemma}
Let $X, Y$ be smooth proper varieties over a field $k$ and
$F: D^b_{coh}(X) \to D^b_{coh}(Y)$ an exact, $k$-linear functor. Suppose
\begin{equation*}
F: \mathrm{Ext}^*_{\mathcal{O}_X} (k(x_1), k(x_2)) \to \mathrm{Ext}^*_{\mathcal{O}_Y} (F(k(x_1)), F(k(x_2)))
\end{equation*}
is an isomorphism for every pair of closed points $x_1, x_2 \in X$. Then $F$ is fully faithful. 
\end{lemma}

\begin{proof} 
By \cite{bondal2003}, $F$ has exact $k$-linear left and right adjoints $R$ and $L$. By our assumption,
\begin{align*}
    \mathrm{Ext}^*_{\mathcal{O}_X} (k(x_1), k(x_2)) &=  \mathrm{Ext}^*_{\mathcal{O}_X} (F(k(x_1)), F(k(x_2))) 
    \\ &= \mathrm{Ext}^* _{\mathcal{O}_X} (k(x_1), RF(k(x_2)))
\end{align*}
for every pair $x_1, x_2$. Since the sheaves $k(x_1)$ form a spanning class, we conclude that the unit $\mathbf{1}_{D^b_{coh}(X)} \to RF$ is an isomorphism on the objects $k(x)$. Now let $K$ be an object of $D^b_{coh}(X)$ and $x \in X$ a closed point. Then
\begin{align*}
    \mathrm{Ext}^*_{\mathcal{O}_X} (LF(K), k(x)) &= \mathrm{Ext}^*_{\mathcal{O}_Y} (F(K), F(k(x))) \\ &= \mathrm{Ext}^*_{\mathcal{O}_X} (K, RF(k(x))) \\
    &=\mathrm{Ext}^*_{\mathcal{O}_X}(K , k(x)),
\end{align*}
induced by the co-unit $LF \to \mathbf{1}_{D^b_{coh} (X)}.$ Again, since the objects $k(x)$ form a spanning class, we conclude that $LF(K) \to K$ is an isomorphism for every $K$, so $F$ is fully faithful.
\end{proof}

\begin{lemma}
Let $X, Y$ be smooth proper varieties over a field $k$. Let $F, G : D^b_{coh} (X) \to D^b_{coh}(Y)$ be exact, $k$-linear functors. Suppose $F$ is fully faithful and $F(k(x)) \cong G(k(x))$ for every closed point $x \in X$. Then the essential image of $G$ is contained in the essential image of $F$.
\end{lemma}

\begin{proof}
By \cite{bondal2003}, $G$ has an exact, $k$-linear right adjoint $R$. Applying the same theorem to $F$ shows that the essential image $\mathcal{A}$ of $F$ is an admissible subcategory of $D^b_{coh} (Y)$. Thus we only have to show that $G(K) \in \prescript{\perp}{}{(\mathcal{A}^\perp)}$ for every $K$ in $D^b_{coh}(X).$ So, suppose $L$ is in $\mathcal{A}^\perp$. Then
\begin{equation}
    \mathrm{Hom} (G(K), L) = \mathrm{Hom} (K, R(L)),
\end{equation}
so it suffices to show $R(L) = 0.$ Applying (6) to the objects $K = k(x)[i]$ gives
\begin{equation*}
    \mathrm{Ext}^*_{\mathcal{O}_X} (k(x), R(L)) = \mathrm{Ext}^*_{\mathcal{O}_Y} (G(k(x)), L) = \mathrm{Ext}^*_{\mathcal{O}_Y} (F(k(x)), L) = 0
\end{equation*}
since $L \in \mathcal{A}^\perp$. But the objects $k(x)$ form a spanning class so it follows that $R(L) = 0$ as needed.
\end{proof}

\begin{prop}
Let $X$ be a smooth proper variety over a field $k$. Let $F: D^b_{coh}(X) \to D^b_{coh}(X)$ be an exact, $k$-linear functor. Suppose $F(\mathcal{F}) \cong \mathcal{F}$ for every coherent sheaf on $X$ with 0-dimensional support. Then $F$ is an equivalence of categories. Furthermore, there exists a line bundle $\mathcal{L}$ on $X$ and a $k$-linear automorphism $f : X \to X$ such that, if we denote by $G$ the functor
\begin{equation*}
    G(K) = \mathbf{L}f^* K \otimes _{\mathcal{O}_X}^\mathbf{L} \mathcal{L},
\end{equation*}
then $F|_{Coh(\mathcal{O}_X)} \cong G|_{Coh(\mathcal{O}_X)}.$
\end{prop}

\noindent \textbf{Remark.} In fact, since $F(k(x)) \cong k(x)$ for every closed point $x \in X$, the automorphism $f$ must be the identity on underlying topological spaces. This implies that $f$ is the identity if $X$ has dimension at least 1 \cite[\href{https://stacks.math.columbia.edu/tag/0G05}{Tag 0G05}]{stacks-project}. This is irrelevant to us though: We care only that $G$ is a Fourier--Mukai auto-equivalence.

\begin{proof}
\underline{Step 1:} $F$ is an equivalence of categories.

Applying Lemma 3 to the functors $\mathbf{1}_{D^b_{coh} (X)}$ and $F$, we see that $F$ is essentially surjective. To prove that $F$ is fully faithful, it suffices by Lemma 2 to prove that
\begin{equation*}
    F: \mathrm{Ext}^*_{\mathcal{O}_X} (k(x), k(y)) \to \mathrm{Ext}^*_{\mathrm{O}_X} (F(k(x)), F(k(y)))
\end{equation*}
is an isomorphism for every pair of closed points $x, y$. It follows from our assumptions that both sides are trivial if $x \neq y$, so there is nothing to prove. If $x = y,$ both sides are isomorphic to 
\begin{equation*}
    \bigwedge_{k(x)} \nolimits^* \mathrm{Ext}^1_{\mathcal{O}_X} (k(x), k(x)),
\end{equation*}
and moreover, $F$ is a $k$-algebra homomorphism. It will therefore suffice to show $F$ is surjective in degrees 0 and 1 (or equivalently injective or bijective). In degree 0, every non-zero homomorphism $k(x) \to k(x)$ is an isomorphism, and functors preserve isomorphisms, so indeed $F$ is injective in degree 0. Suppose given a non-zero extension class $\xi$ of $k(x)$ by $k(x),$ represented by a non-split short exact sequence
\begin{equation*}
    0 \to k(x) \to \mathcal{E} \to k(x) \to 0
\end{equation*}
of coherent sheaves. Then $F(\mathcal{E})$ is a coherent sheaf, and $F(\xi)$ is represented by the short exact sequence
\begin{equation*}
    0 \to F(k(x)) \to F(\mathcal{E}) \to F(k(x)) \to 0
\end{equation*}
of coherent sheaves. Furthermore, we have abstract isomorphisms $F(k(x)) \cong k(x)$ and $F(\mathcal{E}) \cong \mathcal{E}$ by assumption. But now we simply note that an arbitrary extension $\mathcal{E}'$ of $k(x)$ by $k(x)$ is trivial if and only if $\mathcal{E}' \cong k(x) \oplus k(x).$ Thus $F( \mathcal{E}) \cong \mathcal{E} \not \cong k(x) \oplus k(x),$ so $F(\xi) \neq 0$.

\underline{Step 2:} If $\mathcal{E}$ is a locally free $\mathcal{O}_X$-module of rank $r$, then $F(\mathcal{E}[0])$ is also locally free rank $r$ sitting in degree zero. 

Because locally free sheaves of rank $r$ are those objects $K$ of $D^b_{coh}(X)$ such that
\begin{equation*}
    \mathrm{Ext}^*_{\mathcal{O}_X} (K, k(x)) \cong k(x)^{\oplus r} [0]
\end{equation*}
for every closed point $x \in X$.

\underline{Step 3:} If $D$ is an effective divisor on $X$, then $F(\mathcal{O}_D) \in Coh(\mathcal{O}_X)$. 

We have an exact triangle 
\begin{equation*}
    F(\mathcal{O}_X(-D)) \to F(\mathcal{O}_X) \to F(\mathcal{O}_D) \to F(\mathcal{O}_X(-D))[1],
\end{equation*}
thus we see that the cohomology sheaves of $F(\mathcal{O}_D)$ are zero except possibly in degrees $0$ and $-1$, and there is a long exact sequence
\begin{equation*}
    0 \to \mathcal{H}^{-1} (F(\mathcal{O}_D)) \to F(\mathcal{O}_X(-D)) \to F(\mathcal{O}_X) \to \mathcal{H}^{0} (F( \mathcal{O}_D)) \to 0,
\end{equation*}
where by Step 1, $F(\mathcal{O}_X(-D))$ and $F(\mathcal{O}_X)$ are line bundles on $X$. We need to show that $F(\mathcal{O}_X(-D)) \to F(\mathcal{O}_X)$ is injective, but here that's equivalent to
\begin{equation*}
    \mathrm{Supp} \hspace{.5em} \mathcal{H}^{0} (F( \mathcal{O}_D)) \subsetneq X.
\end{equation*}
This follows from 
\begin{align*}
    \mathrm{Hom} _{\mathcal{O}_X} (\mathcal{H}^0 (F(\mathcal{O}_D)), k(x)) &= \mathrm{Hom} _{\mathcal{O}_X} (F(\mathcal{O}_D), k(x)) \\
    &= \mathrm{Hom} _{\mathcal{O}_X} (F(\mathcal{O}_D), F(k(x)))\\
    &= \mathrm{Hom}_{\mathcal{O}_X} (\mathcal{O}_D, k(x))
\end{align*}
for $x \in X$ closed. In fact, this shows $\mathcal{H}^0(F(\mathcal{O}_D))$ and $\mathcal{O}_D$ have the same support.

\underline{Step 4.} Set $\mathcal{L} = F(\mathcal{O}_X).$ Then $F(\mathcal{O}_D) \cong \mathcal{L}|_D$.

We have a short exact sequence of coherent sheaves
\begin{equation*}
    0 \to F( \mathcal{O}_X(-D)) \otimes _{\mathcal{O}_X} \mathcal{L}^{-1} \to \mathcal{O}_X \to F( \mathcal{O}_D) \otimes_{\mathcal{O}_X} \mathcal{L}^{-1} \to 0,
\end{equation*}
and we need to show 
\begin{equation*}
    F( \mathcal{O}_D) \otimes_{\mathcal{O}_X} \mathcal{L}^{-1} \cong \mathcal{O}_D.
\end{equation*}
We know that the left hand side $= \mathcal{O}_{D'}$ for some effective divisor $D',$ and from the proof of Step 3, it follows that $D$ and $D'$ have the same underlying closed subset. Thus $D = \sum n_i D_i$ and $D' = \sum m_i D_i$ with $D_i$ prime divisors and $n_i, m_i$ positive integers. To show that $n_i = m_i,$ let $x$ be a general point of $D_i,$ (so that $D_i$ is regular at $x$ and $x \not \in D_j$ for $j \neq i$). Choose a regular system of parameters $f_1, \dots , f_d$ for $X$ at $x$ such that $D_i = V(f_1)$ in a neighborhood of $x$. Let $\mathcal{F}_N$ be the structure sheaf of the scheme $\mathrm{Spec} (\mathcal{O}_{X,x} / (f_1^{N}, f_2, \dots , f_d)).$ Then we can recover $n_i$ as 
\begin{equation*}
    \dfrac{1}{[k(x): k]}\mathrm{dim}_{k} \mathrm{Hom} _{\mathcal{O}_X}(\mathcal{O}_D,\mathcal{F}_N)
\end{equation*}
for $N \gg 0$. By applying $F$ to $\mathcal{O}_D$ and $\mathcal{F}_N$ we see $n_i = m_i$ so $D = D'$.

\underline{Step 5:} Define a natural isomorphism on line bundles $\mathcal{O}_X(-D), D$ effective.

By Step 3 we have unique isomorphisms $F(\mathcal{O}_X(-D)) \to \mathcal{O}_X(-D) \otimes_{\mathcal{O}_X} \mathcal{L}$ making the diagram
\begin{equation*}
\begin{tikzcd}
     F(\mathcal{O}_X(-D)) \ar[rr] \ar[dr] & &\mathcal{L} \\
      & \mathcal{O}_X(-D) \otimes_{\mathcal{O}_X} \mathcal{L} \ar[ur] &   
\end{tikzcd}
\end{equation*}
commute, where the horizontal arrow is $F$ of the morphism $\mathcal{O}_X(-D) \to \mathcal{O}_X.$ These isomorphisms are natural with respect to inclusions of divisors $D \subset E$.

\underline{Step 6:} $K \in D^b_{coh} (X)$ is in the subcategory $Coh (\mathcal{O}_X)$ if and only if $F(K)$ is.

Let $U$ be an affine open of $X$ with complement $D$ (reduced induced subscheme structure). Applying fully-faithfulness of $F$ and the isomorphisms in Step 5 gives
\begin{align*}
    \Gamma (U, \mathcal{H}^i(F(K))) &=  \mathrm{colim}_{n}\mathrm{Ext}^i_{\mathcal{O}_X} (\mathcal{O}_X(nD), F(K)) \\ &= \mathrm{colim}_{n}\mathrm{Ext}^i_{\mathcal{O}_X}(\mathcal{O}_X (nD), K \otimes _{\mathcal{O}_X}^{\mathbf{L}} \mathcal{L}) \\
    &= \Gamma (U, \mathcal{H}^i (K \otimes _{\mathcal{O}_X}^{\mathbf{L}} \mathcal{L})),
\end{align*}
so $K$ and $F(K)$ have non-zero cohomology sheaves in the same degrees.

\underline{Step 7:} Conclude.

By Step 6, $F$ restricts to a $k$-linear auto-equivalence of the subcategory $Coh(\mathcal{O}_X).$ But these are classified by Gabriel's Theorem (see \cite{gabriel} or \cite[\href{https://stacks.math.columbia.edu/tag/0FZR}{Tag 0FZR}]{stacks-project}): They are all isomorphic to
\begin{equation*}
    \mathcal{F} \mapsto f^*\mathcal{F} \otimes _{\mathcal{O}_X} \mathcal{L}
\end{equation*}
for some $k$-linear automorphism $f$ of $X$ and line bundle $\mathcal{L}$ on $X$, as needed.
\end{proof}

\section{Fully Faithful Functors Agreeing on Coherent Sheaves}

For the following three lemmas, let $X$ be a smooth proper variety of dimension at least 1 over a field $k$. Choose a finite affine open covering $X = \bigcup_i U_i$ such that $\emptyset \subsetneq U_i \subsetneq X$ for every $i$. Let $D_i = X \setminus U_i$ (reduced induced subscheme structure). We will consider the systems of sheaves 
\begin{equation*}
   \cdots \subset \mathcal{O}_X(-nD_i) \subset \mathcal{O}_X(-(n-1)D_i) \subset \cdots \subset \mathcal{O}_X(-D_i).
\end{equation*}
Compare Lemmas A, B, C below to the properties (a), (b), (c) of an ample sequence in \cite{orlov1997equivalences}.

\begin{lemmaa}
Given a coherent sheaf $\mathcal{F}$ on $X$, there exists a surjection from a sheaf of the form
\begin{equation}
\bigoplus _i \mathcal{O}_X(-n_i D_i)^{\oplus r_i}
\end{equation}
to $\mathcal{F}$. Furthermore, we can choose this surjection so that for $m_i \geq n_i$, the induced map \\ $\bigoplus _i \mathcal{O}_X(-m_i D_i)^{\oplus r_i} \to \mathcal{F}$ remains surjective.
\end{lemmaa}

\begin{proof}
Choose a surjection $ \varphi_i : \mathcal{O}_{U_i} ^{\oplus r_i} \to \mathcal{F}|_{U_i},$ then using the formula
\begin{equation*}
    \mathrm{Hom}_{\mathcal{O}_{U_i}} (\mathcal{O}_{U_i} ^{\oplus r_i}, \mathcal{F}|_{U_i}) = \mathrm{colim}_{n_i} \mathrm{Hom}_{\mathcal{O}_X} (\mathcal{O}_X(-n_iD_i)^{\oplus r_i}, \mathcal{F}),
\end{equation*}
we find an $n_i$ and a map $\mathcal{O}_X(-n_iD_i)^{\oplus r_i} \to \mathcal{F}$ which agrees with $\varphi_i$ when restricted to $U_i$. Note that if we increase $n_i,$ the induced map still agrees with $\varphi_i$ on $U_i$. Thus putting these maps together for every $i$ we are done.
\end{proof}

\begin{lemmab}
Assume $K \in D^b_{coh}(X), \mathcal{H}^i(K) = 0$ for $i \geq 0$. Then given a map $\bigoplus _i \mathcal{O}_X(-n_iD_i)^{\oplus r_i}[0] \to K$, there exists $N$ such that the induced map $\bigoplus_i \mathcal{O}_X(-m_iD_i)^{\oplus r_i}[0] \to K$ is zero for $m_i \geq N.$
\end{lemmab}
\begin{proof}
Since
\begin{equation*}
    \mathrm{colim}_N \mathrm{Hom} _{\mathcal{O}_X}(\mathcal{O}_X(-ND_i), K) = \mathrm{Hom}_{ \mathcal{O}_{U_i}}(\mathcal{O}_{U_i}, K|_{U_i}) = 0.
\end{equation*}
\end{proof}

\begin{lemmac}
Let $\mathcal{F}$ be a coherent $\mathcal{O}_X$-module. Then there exists $N$ such that if $n_i \geq N,$
\begin{equation*}
    \mathrm{Hom}_{\mathcal{O}_X} (\mathcal{F}, \bigoplus _i \mathcal{O}_X(-n_i D_i)^{\oplus r_i}) = 0.
\end{equation*}
\end{lemmac}
\begin{proof}
It suffices to show that if $D$ is a non-empty effective divisor on $X$, then
\begin{equation*}
    \mathrm{Hom}_{\mathcal{O}_X} (\mathcal{F}, \mathcal{O}_X(-nD)) = 0
\end{equation*}
for $n \gg 0$. There is by Lemma A a surjection $\bigoplus _i \mathcal{O}_X(-n_iD_i)^{\oplus r_i} \to \mathcal{F}$. Thus it suffices to prove the lemma for $\mathcal{F} = \mathcal{O}_X(-E)$ for some effective divisor $E$. That is, we must show $H^0(X, \mathcal{O}_X(E-nD)) = 0$ for $n\gg 0$. These form a descending sequence of finite dimensional $k$-subspaces of $k(X)$, so it suffices to show the intersection is zero. Let $f$ be in the intersection. Let $Z$ be a prime divisor occurring with positive multiplicity in the divisor $D$. Then $f \in k(X)$ has infinite order of vanishing along $Z$, so $f = 0$.
\end{proof}

For a complex $K$, we will write
\begin{equation*}
    w(K) = \mathrm{sup} \{b-a | \mathcal{H}^b(K) \neq 0 \neq \mathcal{H}^a(K)\}
\end{equation*}
for the \emph{width} of $K$.

\begin{prop} Let $X, Y$ be smooth proper varieties over a field $k$, and let $F, G : D^b_{coh} (X) \to D^b_{coh} (Y)$ be exact, $k$-linear functors. If $F$ is fully faithful and
\begin{equation*}
    F |_{Coh(\mathcal{O}_X)} \cong G |_{Coh(\mathcal{O}_X)},
\end{equation*}
then $F \cong G$.
\end{prop}

\begin{proof}
Lemmas 2 and 3 imply that $G$ is fully faithful with essential image contained in the essential image of $F$. Therefore, $F^{-1} \circ G$ (makes sense and) is fully faithful. Our assumption implies that
\begin{equation*}
   F^{-1} \circ G |_{Coh (\mathcal{O}_X)}
\end{equation*}
is isomorphic to the inclusion $Coh(\mathcal{O}_X) \to D^b_{coh} (X).$ If we can show $F^{-1} \circ G \cong \mathbf{1}_{D^b_{coh} (X)},$ this will imply $F \cong G$ and we will be done. Thus we are reduced to proving the Proposition in the special case $X = Y, G = \mathbf{1}_{D^b_{coh} (X)}.$

Next, if $X$ has dimension 0, i.e., $X = \mathrm{Spec} (k')$ for $k'$ a finite separable extension of $k$, then the result follows from the fact that $D^b_{coh} (X)$ is equivalent to the category of graded $k'$-vector spaces. Thus we may assume $\mathrm{dim} (X) \geq 1$, and we may use Lemmas A, B, C as well as their notation.

For this, we construct by induction on $w = w(K)$ isomorphisms $\alpha_K: K \to F(K)$ for all complexes of width $\leq w$ compatible with all morphisms between such $K$.

The base case is $w(K) = 0$, i.e., $K = \mathcal{F}[n]$ for $\mathcal{F}$ coherent on $X$ and $n \in \mathbf{Z}.$ We are given an isomorphism $\alpha_{\mathcal{F}} : \mathcal{F} \to F(\mathcal{F})$, and we set $\alpha_K = \alpha_{\mathcal{F}} [n]$. One must still check compatibility with all morphisms $K \to K'$ with $w(K) = w(K') = 0$: This is known if $K$ and $K'$ are concentrated in the same degree, but it is not obvious otherwise. However, Orlov's proof of compatibility (see 2.16.4 of \cite{orlov1997equivalences}) carries through unchanged in our context. See also \cite[\href{https://stacks.math.columbia.edu/tag/0FZW}{Tag 0FZW}]{stacks-project}.

\underline{Inductive Step:} We assume given isomorphisms $K \to F(K)$ for all $K$ with $w(K) \leq w-1$ compatible with all morphisms between such $K$. 

\underline{Construction:} Let $K$ be a complex of width $w = b-a$ with cohomology sheaves living in $[a,b]$. There is a surjection from a sheaf $\mathcal{P}$ of the form (7) to $\mathcal{H}^b(K)$. There is an obstruction to lifting this surjection along $K[b] \to \mathcal{H}^b(K)$ which lives in $\mathrm{Ext}^b (\mathcal{P}, \tau_{< b} K).$ Thus by Lemma B, after possibly increasing the integers $n_i$, we can assume (a) the surjection factors through $K[b] \to \mathcal{H}^b(K).$ By Lemma C, after possibly increasing the integers $n_i$ again, we can assume (b) that $\mathrm{Hom}_{\mathcal{O}_X} (\mathcal{H}^b(K), \mathcal{P}) = 0.$ From this it follows that $\mathrm{Hom}(K, \mathcal{P}[-b]) = 0.$ Then there are distinguished triangles as follows:
\begin{equation*}
    \begin{tikzcd}
    \mathcal{P}[-b] \ar[r]\ar[d] &K \ar[r] \ar[d, dashed] \ar[r] &L \ar[r, "+1"] \ar[d] & { } \\
    F(\mathcal{P}[-b]) \ar[r] & F(K) \ar[r] &F(L) \ar[r, "+1"] & { }
    \end{tikzcd}
\end{equation*}
The solid vertical arrows are from the inductive hypothesis (note that $L$ has cohomology sheaves in $[a,b-1]$ by construction), and the dashed arrow exists by the axioms of a triangulated category. It is an isomorphism since the other two vertical arrows are. Since $\mathrm{Hom}(K, \mathcal{P}[-b]) = \mathrm{Hom}(K, F(\mathcal{P}[-b]))  = 0$, the dashed arrow $K \to F(K)$ can be characterized as the unique arrow making the right hand square commute.

\underline{Independence of Surjection:} If $\mathcal{P}'\to \mathcal{H}^b(K)$ is another surjection where $\mathcal{P}'$ has the form (7) and (a) and (b) are satisfied, we want to show that we get the same definition of the arrow $K \to F(K).$ Note that $\mathcal{P} \oplus \mathcal{P}' \to \mathcal{H}^b(K)$ also satisfies (a) and (b) so we may assume the map $\mathcal{P}'[-b] \to K$ factors through $\mathcal{P}[-b] \to K.$ Choose a distinguished triangle $\mathcal{P}'[-b'] \to K \to L' \to \mathcal{P}'[-b+1],$ and write $\alpha, \alpha'$ for the arrow constructed using $\mathcal{P}, \mathcal{P}'$, respectively. Then $K \to L$ factors through $K \to L'$ since by the axioms of a triangulated category there is a fill in:
\begin{equation*}
    \begin{tikzcd}
         \mathcal{P}'[-b] \ar[r] \ar[d] &K \ar[r] \ar[d, "="] & L' \ar[r, "+1"] \ar[d, dashed] & {} \\
         \mathcal{P}[-b] \ar[r] & K \ar[r] &L \ar[r, "+1"] &{}
    \end{tikzcd}
\end{equation*}
But then by the definition of $\alpha'$ and the inductive hypothesis, the diagram 
\begin{equation*}
    \begin{tikzcd}
         K \ar[r] \ar[d, "\alpha'"] &L' \ar[r] \ar[d] &L \ar[d] \\
         F(K) \ar[r] &F(L') \ar[r] & F(L)
    \end{tikzcd}
\end{equation*}
commutes. Thus by uniqueness of $\alpha$, $\alpha = \alpha'$.

\underline{Compatibility:} We need to show the isomorphisms constructed are compatible with all morphisms $K \to K'$ between objects of width at most $w$. We do this by induction on $w(K)+w(K').$ Let $K$ have cohomology sheaves in degrees $[a,b]$ and $K'$ in degrees $[a',b'].$ Note that all cases with $w(K), w(K') < w$ are known by the first induction hypothesis.

\underline{Case 1:} $b'< b$.

Choose a triangle $\mathcal{P}[-b] \to K \to L \to \mathcal{P}[-b+1]$ as in the construction. Composing with $K \to K'$ we get a map $\mathcal{P}[-b] \to K'$. Then using Lemma B we may assume by increasing the $n_i$ that this map is zero. Then $K \to K'$ factors through $L.$ Since $L$ has cohomology sheaves in $[a,b-1],$ our  inductive hypotheses imply that the right square
\begin{equation*}
    \begin{tikzcd}
    K \ar[r] \ar[d] &L \ar[r] \ar[d] &K'\ar[d] \\
    F(K) \ar[r] &F(L) \ar[r] &F(K')
    \end{tikzcd}
\end{equation*}
is commutative, while the left square commutes by construction. So we're done in this case.

\underline{Case 2:} $b' \geq b$.

Choose a distinguished triangle $\mathcal{P}'[-b'] \to K' \to L' \to \mathcal{P}'[-b+1]$ as in the construction (for $K'$ instead of $K$) with the property that $\mathrm{Hom}(\mathcal{H}^{b'}(K), \mathcal{P}') = 0$. This is possible by Lemma C. Then by our choice of $\mathcal{P}',$ 
$\mathrm{Hom}(K, K') \to \mathrm{Hom} (K, L')$ is injective, and since $F(K) \cong F(K')$ and $F$ is fully faithful, $\mathrm{Hom}(K, F(K')) \to \mathrm{Hom} (K, F(L'))$ is injective as well. Consider the diagram:
\begin{equation*}
    \begin{tikzcd}
    K \ar[r] \ar[d] &K' \ar[r] \ar[d] &L' \ar[d] \\
    F(K) \ar[r] & F(K')  \ar[r] & F(L') 
    \end{tikzcd}
\end{equation*}
The right square commutes by construction. Thus by the injectivity on Hom's, it suffices to show the outer square commutes. But this follows from our induction hypotheses since $L'$ has cohomology sheaves in $[a', b'-1].$
\end{proof}

\begin{proof}[Proof of Theorem]
By Proposition 1, there is $E \in D^b_{coh} (X \times Y)$ such that $F(\mathcal{F}) \cong \Phi_E(\mathcal{F})$ for every coherent $\mathcal{O}_X$-module $\mathcal{F}$ with 0-dimensional support. By Lemma 3, the essential image of $\Phi_E$ is contained in the essential image of $F$, thus the functor $F^{-1} \circ \Phi_E$ makes sense. By combining Propositions 2 and 3, we obtain an isomorphism $F^{-1} \circ \Phi_E \cong G$ where $G$ is a Fourier--Mukai auto-equivalence of $D^b_{coh} (X)$. But then $F \cong \Phi_E \circ G^{-1}$, which is Fourier--Mukai.
\end{proof}

\newpage

\bibliographystyle{alpha}
\bibliography{references}

\end{document}